\documentclass[12pt,notitlepage]{amsart}
\usepackage{latexsym,amsfonts,amssymb,amsmath,amsthm} 
\newcommand{\mz}{\ensuremath{\mathbb Z}}
\newcommand{\mr}{\ensuremath{\mathbb R}}

\newcommand{\shortmod}{\ensuremath{\negthickspace \negthickspace \negthickspace \pmod}}

\newcommand{\half}{\ensuremath{ \frac{1}{2}}}

\newcommand{\intR}{\int_{-\infty}^{\infty}}

\newcommand{\thalf}{\tfrac12}
\newcommand{\sumstar}{\sideset{}{^*}\sum}
\newcommand{\leg}[2]{\left(\frac{#1}{#2}\right)}
\newcommand{\e}[2]{e\left(\frac{#1}{#2}\right)}
\newcommand{\lp}{\left(}
\newcommand{\rp}{\right)}

\theoremstyle{plain}		
	\newtheorem{mytheo}{Theorem}[section]
	
	\newtheorem{myprop}[mytheo]{Proposition}
	
     \newtheorem{mylemma}[mytheo]{Lemma}
	
	\newtheorem{myconj}[mytheo]{Conjecture}
	\newtheorem{myremark}[mytheo]{Remark}

\theoremstyle{remark}

\numberwithin{equation}{section}
\begin{document}
\title{The first moment of quadratic Dirichlet L-functions}
\author{Matthew P. Young} 
\address{Department of Mathematics \\
	  Texas A\&M University \\
	  College Station \\
	  TX 77843-3368 \\
		U.S.A.}
\email{myoung@math.tamu.edu}

\begin{abstract}
We obtain an asymptotic formula for the smoothly weighted first moment of primitive quadratic Dirichlet L-functions at the central point, with an error term that is ``square-root'' of the main term.  Our approach uses a recursive technique that feeds the result back into itself, successively improving the error term.
\end{abstract}

\maketitle
\section{Introduction}
Quadratic twists of L-functions have been widely studied by many people; some of the motivating problems involve ranks of elliptic curves, and connections to Landau-Siegel zeros.  In this article we consider the first moment of the family of quadratic Dirichlet L-functions.  Our main result is
\begin{mytheo}
\label{theo:mainthm}
Let $\Phi:\mr^{+} \rightarrow \mr$ be a smooth function of compact support.  Then
\begin{equation}
\label{eq:mainthm}
 \sumstar_{(d,2)=1} L(\thalf, \chi_{8d}) \Phi\leg{d}{X} =  X P(\log{X}) + O(X^{1/2 + \varepsilon}),
\end{equation}
for some linear polynomial $P$ (depending on $\Phi$).  
\end{mytheo}
Jutila \cite{Jutila} and Vinogradov-Takhtadzhyan \cite{VT} obtained a similar result but with an error term of size $X^{3/4 + \varepsilon}$, and Goldfeld-Hoffstein \cite{Goldfeld} obtained an error of size $X^{19/32 + \varepsilon}$, using properties of certain Eisenstein series (one difference to note is that these authors considered the unsmoothed moment).  A. Kontorovich has informed me that an error term of size $O(X^{1/2 + \varepsilon})$ is essentially implicit in \cite{Goldfeld}, so that Theorem \ref{theo:mainthm} does not constitute an improvement of the error term in the smoothed version of the first moment.  However, our method seems to have some novelty and can almost surely be used to improve the error terms in other moment problems involving primitive quadratic twists.

Our approach is most similar to the work of Soundararajan \cite{Sound}, who obtained an asymptotic formula with a power savings for the third moment of this family.  Indeed, the proof of Theorem \ref{theo:mainthm} required significant ``off-diagonal'' analysis of the same type performed by Soundararajan.  Actually, we required two new ingredients not found in \cite{Sound}: most importantly, we used recursion to successively improve the exponent in the error term which in the limit converges to $1/2$.  Secondly, we performed an intricate analysis of certain subsidiary terms that can be rather easily bounded by $O(X^{3/4 + \varepsilon})$, but seem to rely on the Riemann Hypothesis to be improved (on an individual basis); fortunately, it turns out that these terms all cancel out! (See Section \ref{section:mainterms}.)

The use of recursion to study mean values of real characters was fruitfully applied by Heath-Brown in his work on the large sieve for quadratic characters \cite{H-B}.  One of the primary difficulties when studying mean values of quadratic characters (and more generally, characters of fixed order) is treating the squarefree (i.e. primitivity) condition.  Heath-Brown succeeded at this problem in \cite{H-B} by relaxing the condition that $d$ is squarefree to the condition that the square divisors of $d$ are small (note that this uses positivity and the fact that he is obtaining an upper bound and not an asymptotic formula).  In a sense, the use of recursion in the proof of Theorem \ref{theo:mainthm} is what allows us to treat the squarefree condition in an essentially optimal way.  Soundararajan (private communication) has also indicated to us that an iterative method can potentially be useful in these moment problems.  The paper \cite{BY} also used recursion in the study of the first moment of cubic Dirichlet L-functions, but in that work we did not identify the cancellation between various subsidiary terms.  

Goldfeld and Hoffstein \cite{Goldfeld} point out that it seems unlikely to improve \eqref{eq:mainthm} without some substantial improvement in the zero-free region for the Riemann zeta function, simply due to the state of knowledge on the distribution of square-free integers.


\section{Preliminaries}
Let $\chi_d$ be the quadratic Dirichlet character associated to the fundamental discriminant $d$.  We shall work with discriminants of the form $8d$ where $d$ is odd, squarefree, and positive, so that $\chi_{8d}$ is an even, primitive character of conductor $8d$.

The functional equation for such quadratic Dirichlet L-functions reads
\begin{equation}
\Lambda(s, \chi_{8d}) = \left(\frac{8d}{\pi}\right)^{s/2} \Gamma\left(\frac{s}{2} \right) L(s, \chi_{8d}) = \Lambda(1-s,\chi_{8d}),
\end{equation}
or in its asymmetric form
\begin{equation}
L(s, \chi_{8d}) = X(s) L(1-s, \chi_{8d}),
\end{equation}
where
\begin{equation}
X(\thalf + u) = \left(\frac{8d}{\pi} \right)^{-u } \frac{\Gamma\left(\frac{1/2 - u}{2} \right)}{\Gamma\left(\frac{1/2 + u }{2} \right)}.
\end{equation}

\subsection{Approximate functional equation}
Our basic formula to represent a value of the $L$-functions inside the critical strip is the following
\begin{myprop}[Approximate functional equation]
Let $G(s)$ be an entire, even function, bounded in any strip $-A \leq \text{Re}(s) \leq A$.  Then
\begin{equation}
L(\thalf + \alpha, \chi)
=
\sum_{n} \frac{\chi_{8d}(n) }{n^{\thalf + \alpha}} V_{\alpha} \left(\frac{n}{d^{1/2}}\right)
+ X_{\alpha} \sum_n \frac{\chi_{8d}(n)}{n^{\thalf - \alpha}} V_{-\alpha} \left(\frac{n}{d^{1/2}}\right),
\end{equation}
where $V$, $g$, and $X_{\alpha}$ are given by the following:
\begin{equation}
\label{eq:V}
V_{\alpha}(x) = \frac{1}{2 \pi i} \int_{(1)} \frac{G(s)}{s} g_{\alpha}(s) x^{-s} ds,
\end{equation}
\begin{equation}
g_{\alpha}(s) = \left(\frac{8}{\pi}\right)^{\frac{s}{2}}
\frac{\Gamma\left(\frac{\half + \alpha + s}{2} \right)}{\Gamma\left(\frac{\half + \alpha}{2} \right)} 
\end{equation}
and
\begin{equation}
 X_{\alpha} = \left(\frac{8d}{\pi} \right)^{-\alpha} 
\frac{\Gamma\left(\frac{\thalf- \alpha}{2} \right)}{\Gamma\left(\frac{\thalf + \alpha}{2} \right)}.
\end{equation}
\end{myprop}
\begin{myremark}
\label{remark:zero}
 We also choose $G$ so that $G(\pm \alpha) = G(\thalf \pm \alpha) = 0$; for definiteness, one can take
\begin{equation}
 G(s) = e^{s^2} \frac{(s^2 + \alpha^2)((s-\thalf)^2 - \alpha^2)((s+\thalf)^2 - \alpha^2)}{\alpha^2 (\frac14 - \alpha^2)^2}.
\end{equation}

\end{myremark}
The purpose of this Remark is to simplify certain future calculations (these zeros may cancel certain poles that arise).

\subsection{Poisson summation formula}
We now quote Soundararajan's result (see the proof of Lemma 2.6 of \cite{Sound}).
\begin{mylemma}
\label{lemma:Poisson}
 Let $F$ be a smooth function with compact support on the positive real numbers, and suppose that $n$ is an odd integer.  Then
\begin{equation}
 \sum_{(d,2) = 1} \leg{d}{n} F\left(\frac{d}{Z}\right) = \frac{Z}{2n} \leg{2}{n} \sum_{k \in \mz} (-1)^k G_k(n) \widehat{F}\left(\frac{kZ}{2n}\right),
\end{equation}
where
\begin{equation}
 G_k(n) = \left(\frac{1-i}{2} + \leg{-1}{n} \frac{1+i}{2}\right) \sum_{a \shortmod{n}} \leg{a}{n} \e{ak}{n},
\end{equation}
and
\begin{equation}
 \widehat{F}(y) = \intR (\cos(2 \pi x y) + \sin(2 \pi x y)) F(x) dx
\end{equation}
is a Fourier-type transform of $F$.
\end{mylemma}

The Gauss-type sum is calculated exactly with the following (which is Lemma 2.3 of \cite{Sound}).
\begin{mylemma}
\label{lemma:Gk}
If $m$ and $n$ are relatively prime odd integers, then $G_k(mn) = G_k(m) G_k(n)$, and if $p^{\alpha}$ is the largest power of $p$ dividing $k$ (setting $\alpha=\infty$ if $k=0$), then
\begin{equation}
\label{eq:Gauss}
 G_k(p^{\beta}) = 
\begin{cases}
 0, \qquad & \text{if $\beta \leq \alpha$ is odd}, \\
 \phi(p^{\beta}), \qquad & \text{if $\beta \leq \alpha$ is even}, \\
 -p^{\alpha}, \qquad & \text{if $\beta = \alpha + 1$ is even}, \\
 \leg{kp^{-\alpha}}{p} p^{\alpha} \sqrt{p}, \qquad & \text{if $\beta = \alpha +1$ is odd}, \\
 0, \qquad & \text{if $\beta \geq \alpha +2$.}
\end{cases}
\end{equation}
\end{mylemma}

\section{Setting up the problem}
We generalize the basic quantity of study \eqref{eq:mainthm}, and consider
\begin{equation}
  M(\alpha,l) = \sumstar_{(d,2)=1} \chi_{8d}(l) L(\thalf + \alpha, \chi_{8d}) \Phi\left(\frac{d}{X}\right)
\end{equation}
for $l$ odd squarefree.  This can be thought of as a ``twisted'' moment, which appears naturally when mollifying and amplifying central values.  Our analysis of $M(0,1)$ leads to general quantities $M(\alpha,l)$, so our method requires to study these more general expressions.

\subsection{Averaging the approximate functional equation}
Using the approximate functional equation, write $M(\alpha,l) = M_1(\alpha,l) + M_2(\alpha,l)$, where
\begin{equation}
 M_1(\alpha,l) = \sumstar_{(d,2) = 1} \Phi\left(\frac{d}{X}\right) \sum_{n} \frac{\chi_{8d}(nl) }{n^{\thalf + \alpha}} V_{\alpha} \left(\frac{n}{d^{1/2}}\right),
\end{equation}
and
\begin{equation}
 M_2(\alpha,l) = \sumstar_{(d,2) = 1} \Phi\left(\frac{d}{X}\right) X_{\alpha} \sum_{n} \frac{\chi_{8d}(nl) }{n^{\thalf-\alpha}} V_{-\alpha} \left(\frac{n}{d^{1/2}}\right).
\end{equation}

We can reduce $M_2$ to a version of $M_1$ but with slightly different parameters, namely
\begin{equation}
 M_2(\alpha,l) = \gamma_{\alpha} X^{-\alpha} \sumstar_{(d,2) = 1} \Phi_{-\alpha}\left(\frac{d}{X}\right) \sum_{n} \frac{\chi_{8d}(nl) }{n^{\thalf-\alpha}} V_{-\alpha} \left(\frac{n}{d^{1/2}}\right),
\end{equation}
where $\Phi_{s}(x) = x^{s} \Phi(x)$ and
\begin{equation}
 \gamma_{\alpha} = \left(\frac{8}{\pi} \right)^{-\alpha} 
\frac{\Gamma\left(\frac{\thalf- \alpha}{2} \right)}{\Gamma\left(\frac{\thalf + \alpha}{2} \right)}.
\end{equation}
For ease of reference, we state this simple development as follows.
\begin{myremark}
\label{remark:M1toM2}
To derive an expression for $M_2$ via a corresponding term from $M_1$ involves swapping $\alpha$ and $-\alpha$, replacing $\Phi(x)$ by $\Phi_{-\alpha}(x) = x^{-\alpha}\Phi(x)$, and multiplying by $\gamma_{\alpha} X^{-\alpha}$ (in that order).  Note that this operation is an involution.
\end{myremark}

As a convention, we shall often not write the dependence on $\alpha$ and $l$ and instead simply write $M_1$, $M_2$, etc.

\subsection{The conjecture}
For convenience, we state
\begin{myconj}
The conjecture of \cite{CFKRS} and \cite{DGH} is
\begin{multline}
 \sumstar_{(d,2)=1} L(\thalf + \alpha, \chi_{8d}) \Phi\left(\frac{d}{X}\right) 
= \frac{X \widetilde{\Phi}(1)}{2 \zeta_2(2)} \zeta_{2}(1 + 2\alpha) B_{\alpha} 
\\
+ \frac{X^{1-\alpha} \widetilde{\Phi}(1-\alpha) \gamma_{\alpha}}{2 \zeta_2(2)} \zeta_{2}(1 - 2\alpha) B_{-\alpha}
+ O(X^{1/2 + \varepsilon}),
\end{multline}
where $\widetilde{\Phi}$ is the Mellin transform of $\Phi$, and
\begin{equation}
 \zeta_2(1 + 2\alpha) B_{\alpha} := \sum_{(n,2)=1} \frac{1}{n^{1 + 2\alpha}} \prod_{p | n} (1 + p^{-1})^{-1}.
\end{equation}
Here $B_{\alpha}$ has an absolutely convergent Euler product for $\alpha$ in a neighborhood of the origin. Furthermore, the estimate holds uniformly for $|\text{Re}(\alpha)| \ll (\log{X})^{-1}$ and $|\text{Im}(\alpha)| \ll (\log{X})^2$, say.
\end{myconj}

One can derive this conjecture from the recipe of \cite{CFKRS} using the orthogonality relation
\begin{equation}
 \sumstar_{(d,2) = 1} \chi_{8d}(m) \Phi \leg{d}{X} \sim \frac{X \widetilde{\Phi}(1)}{2 \zeta_2(2)} \prod_{p|m} (1+ p^{-1})^{-1},
\end{equation}
for $m$ an odd square, and is $o(X)$ otherwise (for fixed $m$).

Actually we need the ``twisted'' moment conjecture (see \cite{HY}, Section 7), which states the following
\begin{myconj} 
\label{conj:twistedconj}
We have for any odd squarefree $l$ that the following holds uniformly for $|\text{Re}(\alpha)| \ll (\log{X})^{-1}$ and $|\text{Im}(\alpha)| \ll (\log{X})^2$
\begin{multline}
\label{eq:Malphal}
M(\alpha,l)
= \frac{X \widetilde{\Phi}(1)}{2 \zeta_2(2)} l^{-1/2 - \alpha} \zeta_{2}(1 + 2\alpha) B_{\alpha}(l) 
\\
+ \frac{X^{1-\alpha} \widetilde{\Phi}(1-\alpha) \gamma_{\alpha}}{2 \zeta_2(2)} l^{-1/2 + \alpha} \zeta_{2}(1 - 2\alpha) B_{-\alpha}(l)  + O((lX)^{1/2 + \varepsilon}),
\end{multline}
where
\begin{equation}
\zeta_2(1 + 2\alpha) B_{\alpha}(l) = \sum_{(n,2)=1} \frac{1}{n^{1 + 2\alpha}} \prod_{p | nl} (1 + p^{-1})^{-1}.
\end{equation}
\end{myconj}
We shall prove Conjecture \ref{conj:twistedconj} with the stated error term by use of the following recursive result
\begin{mytheo}
\label{theo:recursive}
Suppose that Conjecture \ref{conj:twistedconj} holds but with an error term of size $l^{1/2 + \varepsilon} X^{f + \varepsilon}$, for some $1/2 \leq f \leq 1$.  Then Conjecture \ref{conj:twistedconj} holds with an error of size $l^{1/2 + \varepsilon} X^{\frac{f+\half}{2} + \varepsilon}$.
\end{mytheo}
We will obtain the exponent $f_0=1$ as an initial estimate, which leads to the sequence $f_1 = 3/4$, $f_2 = 5/8, \dots,$ where clearly $\lim_{n \rightarrow \infty} f_n = 1/2$, and whence Conjecture \ref{conj:twistedconj} follows.
Specializing to $\alpha =0$, $l=1$ gives Theorem \ref{theo:mainthm}.

In what follows we assume that $|\alpha| \gg (\log{X})^{-1}$, $|\text{Im}(\alpha)| \ll \log^2{X}$, and $|\text{Re}(\alpha)| \ll (\log{X})^{-1}$ in order to claim uniformity in terms of $\alpha$; note that $\alpha$ being close to zero is problematic in two ways, namely because of the pole of $\zeta(1 + 2 \alpha)$, and because of the uniformity of $G(s)$ in terms of $\alpha$.  Once we establish \eqref{eq:Malphal} uniformly for such $\alpha$, we extend the result to $|\alpha| \ll (\log{X})^{-1}$ with the same quoted error term in the following way.  Since $M(\alpha,l)$ and the main term of \eqref{eq:Malphal} are holomorphic for $\alpha$ near $0$, we see that the error term must also be holomorphic in terms of $\alpha$.  Then apply the maximum modulus principle to the error term, with respect to the disk $|\alpha| \ll (\log{X})^{-1}$.

\subsection{A short calculation of $B_{\alpha}(l)$}
In order to recognize certain other expressions in terms of $B_{\alpha}$, and also to observe the absolute and uniform convergence of $B_{\alpha}(l)$ for $\alpha$ in a neighborhood of the origin, we now state
\begin{mylemma}
 We have
\begin{equation}
\label{eq:Balpha2}
 B_{\alpha}(l) = \prod_{p | l} (1 + p^{-1})^{-1}  \prod_{p \nmid 2l} \lp 1 - p^{-2-2\alpha} (1+p^{-1})^{-1} \rp,
\end{equation}
and
\begin{equation}
\label{eq:Balpha}
\frac{1}{\zeta_2(2)} B_{\alpha}(l) =  \frac{\phi(l)}{l} \prod_{p \nmid 2l} (1 - p^{-2} -p^{-2-2\alpha} + p^{-3-2\alpha}).
\end{equation}
\end{mylemma}
\begin{proof}

We have
\begin{align}
\zeta_2(1 + 2\alpha)  B_{\alpha}(l) &= \sum_{(n,2)=1} \frac{1}{n^{1 + 2\alpha}} \prod_{p | nl} (1 + p^{-1})^{-1}
\\
&= \prod_{p | l} (1 + p^{-1})^{-1} \sum_{(n,2)=1} \frac{1}{n^{1 + 2\alpha}} \prod_{p | n, p \nmid l} (1 + p^{-1})^{-1}
\\
&= \prod_{p | l} (1 + p^{-1})^{-1}  (1-p^{-1-2\alpha})^{-1} \sum_{(n,2l)=1} \frac{1}{n^{1 + 2\alpha}} \prod_{p | n} (1 + p^{-1})^{-1},
\end{align}
which upon performing the sum over $n$ gives
\begin{align}
\zeta_2(1 + 2\alpha)  B_{\alpha}(l) &= \prod_{p | l} (1 + p^{-1})^{-1}  (1-p^{-1-2\alpha})^{-1} \prod_{p \nmid 2l} \lp 1+ \frac{p^{-1-2\alpha}}{1- p^{-1-2\alpha}} (1+p^{-1})^{-1} \rp
\\
&= \zeta_{2}(1 + 2\alpha) \prod_{p | l} (1 + p^{-1})^{-1}  \prod_{p \nmid 2l} \lp 1 - p^{-1-2\alpha} + p^{-1-2\alpha} (1+p^{-1})^{-1} \rp
\\
&= \zeta_{2}(1 + 2\alpha) \prod_{p | l} (1 + p^{-1})^{-1}  \prod_{p \nmid 2l} \lp 1 - p^{-2-2\alpha} (1+p^{-1})^{-1} \rp,
\end{align}
which is \eqref{eq:Balpha2}.
Continuing, we also have
\begin{align}
 \frac{1}{\zeta_2(2)} B_{\alpha}(l) &= \prod_{p | l} (1 + p^{-1})^{-1}  \prod_{p \nmid 2l} \lp 1 - p^{-2-2\alpha} (1+p^{-1})^{-1} \rp \prod_{p \nmid 2} (1-p^{-2})
\\
&= \prod_{p | l} (1- p^{-1}) \prod_{p \nmid 2l} (1-p^{-2})  \lp 1 - p^{-2-2\alpha} (1+p^{-1})^{-1} \rp
\\
&= \frac{\phi(l)}{l} \prod_{p \nmid 2l} (1 - p^{-2} -p^{-2-2\alpha} + p^{-3-2\alpha}).
\qedhere
\end{align}
\end{proof}

\subsection{Removing the squarefree condition}
Our basic strategy is to employ Poisson summation to the sum over $d$.  Of course, one must remove the squarefree condition in some way.  Naturally we shall use M\"{o}bius inversion.  In this way we obtain
\begin{equation}
 M_1 = \sum_{(a,2l) = 1} \mu(a) \sum_{(d,2)=1} \Phi\left(\frac{d a^2}{X}\right) \sum_{(n,2a)=1} \frac{\chi_{8d}(nl) }{n^{\thalf + \alpha}} V_{\alpha} \left(\frac{n}{a d^{1/2}}\right).
\end{equation}
Now we separate the terms with $a \leq Y$ and with $a > Y$ ($Y$ a parameter to be chosen later), writing $M_1 = M_N + M_R$, respectively.  We shall treat these two terms with completely different methods.  We also set the notation $M_2 = M_{-N} + M_{-R}$.

\subsection{Outline of the rest of the paper}
We compute $M_R$ in Section \ref{section:MR}, and $M_N$ in Section \ref{section:MN}.  These are completely different computations, as we use Poisson summation on $M_N$, while we reduce $M_R$ to an expression similar to the original moment $M(\alpha,l)$ to which we extract a main term using Conjecture \ref{conj:twistedconj} but with an error term of size $O(l^{1/2 + \varepsilon} X^{f + \varepsilon})$.  This analysis naturally expresses each of $M_R$ and $M_N$ as the sum of certain contour integrals which are individually difficult (if not impossible) to estimate to our desired degree of accuracy, plus an acceptable error term.  For example, the familiar main term obtained by taking the $k=0$ term after Poisson summation (here $k$ is the ``dual'' variable) is given, after some simplifications, by \eqref{eq:P1}.  A standard contour shift and estimation of the tail of the sum over $a$ easily expresses this term as one of the main terms of Conjecture \ref{conj:twistedconj}, plus an error of size $O(X^{3/4 + \varepsilon}) + O(X/Y)$ (taking $l=1$ for simplicity).  The other terms are similar, although some do not contribute a main term.  Although we cannot estimate each of the various contour integrals with an error better than $O(X^{3/4 + \varepsilon})$, it turns out that there is a lot of simplification that occurs by summing the integrals; we exhibit this pleasant behavior in Section \ref{section:mainterms}.  

\section{Estimating $M_R$}
\label{section:MR}
Our plan in analyzing $M_R$ is to go back to squarefree integers but where the variables are significantly shorter than before (assuming $Y$ is moderately large), an idea that seems to go back to Iwaniec \cite{Iwaniec}.  We shall use a refinement that was recently used in \cite{BY}, whereby one goes a step further and uses a known estimate for $M(\alpha,l)$ (in this case an asymptotic formula) to obtain an improved result.  In this article we will extract a kind of main term that naturally combines with the main term of $M_N$.  In this manner we can vastly improve the size of the error term in this analysis.
\subsection{Reintroducing squarefrees}
To go back to squarefree integers, let $d \rightarrow b^2 d$, which gives
\begin{equation}
M_R = \sumstar_{(d,2)=1}  \sum_{(b,2l)=1} \sum_{\substack{(a,2l) = 1 \\ a > Y}} \mu(a) \Phi\left(\frac{d (ab)^2}{X}\right) \sum_{(n,ab)=1} \frac{\chi_{8d}(nl)}{n^{\thalf + \alpha}} V_{\alpha} \left(\frac{n}{ab \sqrt{d}}\right).
\end{equation}
Letting $c=ab$ be a new variable, we get
\begin{equation}
M_R = \sum_{(c,2l)=1} \lp \sum_{\substack{a |c \\ a > Y}} \mu(a) \rp \sumstar_{(d,2)=1}  \Phi\left(\frac{d c^2}{X}\right) \sum_{(n,c)=1} \frac{\chi_{8d}(nl)}{n^{\thalf + \alpha}} V_{\alpha} \left(\frac{n}{c \sqrt{d}}\right).
\end{equation}
Now use the integral representation of $V_{\alpha}$ (i.e. \eqref{eq:V}) to get
\begin{equation}
M_R = \sum_{(c,2l)=1} \lp \sum_{\substack{a |c \\ a > Y}} \mu(a) \rp \sumstar_{(d,2)=1}  \Phi\left(\frac{d c^2}{X}\right) \sum_{(n,c)=1} \frac{\chi_{8d}(nl)}{n^{\thalf + \alpha}} \frac{1}{2 \pi i} \int_{(\thalf + \varepsilon)} \leg{c \sqrt{d}}{n}^s \frac{G(s)}{s} g_{\alpha}(s) ds.
\end{equation}
Moving the sum over $n$ inside the integral, we get
\begin{multline}
M_R = \sum_{(c,2l)=1} \lp \sum_{\substack{a |c \\ a > Y}} \mu(a) \rp \sumstar_{(d,2)=1} \chi_{8d}(l) \Phi\left(\frac{d c^2}{X}\right) 
\\
\frac{1}{2 \pi i} \int_{(\thalf + \varepsilon)} (c \sqrt{d})^s L(\thalf + \alpha + s, \chi_{8d} \chi_{0,c}) \frac{G(s)}{s} g_{\alpha}(s) ds,
\end{multline}
where $\chi_{0,c}$ is the principal character to modulus $c$.  Now move the line of integration to $\varepsilon$ without crossing any poles in this process (by Remark \ref{remark:zero}).  Then express the Dirichlet $L$-function in terms of its associated primitive one, getting
\begin{multline}
\label{eq:prerecursion}
M_R = \sum_{(c,2l)=1} \lp \sum_{\substack{a |c \\ a > Y}} \mu(a) \rp  \sum_{r | c} \frac{\mu(r)}{r^{1/2 + \alpha}} \sumstar_{(d,2)=1} \chi_{8d}(lr)  \Phi\left(\frac{d c^2}{X}\right) 
\\
\frac{1}{2 \pi i} \int_{(\varepsilon)} \leg{c \sqrt{d}}{r}^s L(\thalf + \alpha + s, \chi_{8d}) \frac{G(s)}{s} g_{\alpha}(s) ds,
\end{multline}
where $\varepsilon \asymp (\log{X})^{-1}$.

\subsection{Using the recursion}
Observe that the sum over $d$ takes the form
\begin{equation}
 \sumstar_{(d,2)=1} \chi_{8d}(lr) \Phi\left(\frac{d c^2}{X}\right) (c \sqrt{d})^s L(\thalf + \alpha + s, \chi_{8d}) = X^{\frac{s}{2}} \sumstar_{(d,2)=1} \chi_{8d}(lr) \Phi_{s/2}\left(\frac{d}{X'}\right) L(\thalf + \alpha + s, \chi_{8d}),
\end{equation}
where recall $\Phi_{s}(x) = x^{s} \Phi(x)$ and $X' = X/c^2$.  At this point we apply our inductive hypothesis, namely that
\begin{equation}
 \sumstar_{(d,2)=1} \chi_{8d}(lr) \Phi_{s/2} \left(\frac{d}{X'}\right) L(\thalf + \alpha + s, \chi_{8d}) = M.T. + O((lr)^{1/2 + \varepsilon} {X'}^{f + \varepsilon}),
\end{equation}
where $f \geq \half$ is some constant, and $\text{Im}(s) \ll \log^2{X}$.  Note that due to the exponential decay of $g_{\alpha}(s)$ as $|\text{Im}(s)| \rightarrow \infty$, we may truncate the integral appearing in \eqref{eq:prerecursion} for $\text{Im}(s) \ll \log^2{X}$ at no cost.

As the basis step one can take $f=1$ due to known estimates for the second moment of this family due to Jutila \cite{Jutila} (with only slightly more work the treatment in this article can also establish this initial estimate for the first moment).  Here the $M.T.$ is given by the main term appearing in Conjecture \ref{conj:twistedconj} (note that to establish the intial case, using $f=1$, no knowledge of $M.T.$ is requiried other than $M.T. \ll X^{1 + \varepsilon}$).  To be completely explicit, we have
\begin{multline}
M.T. = \frac{X}{c^2} \frac{\widetilde{\Phi_{s/2}}(1)}{2 \zeta_2(2)} (lr)^{-1/2-\alpha-s} \zeta_2(1 + 2\alpha +2s) B_{\alpha +s}(lr)
\\
+  \leg{X}{c^2}^{1-\alpha-s} \frac{\widetilde{\Phi_{s/2}}(1-\alpha-s)}{2 \zeta_2(2)} \gamma_{\alpha+s} (lr)^{-1/2+\alpha+s} \zeta_2(1 - 2\alpha -2s) B_{-\alpha -s}(lr).
\end{multline}
Note $\widetilde{\Phi_{s/2}}(1) = \widetilde{\Phi}(1+\tfrac{s}{2})$ and $\widetilde{\Phi_{s/2}}(1-\alpha-s) = \widetilde{\Phi}(1-\alpha-\tfrac{s}{2})$

Inserting this calculation into the computation for $M_R$ gives that $M_R = M_{R1} + M_{R2} + E.T.$, say, where we compute (using the assumption $f \geq 1/2$)
\begin{equation}
 E.T. \ll X^{\varepsilon} \sum_{(c,2l)=1} \lp \sum_{\substack{a |c \\ a > Y}} 1 \rp  \sum_{r | c} \frac{1}{r^{1/2}} (lr)^{1/2 + \varepsilon} \leg{X}{c^2}^{f + \varepsilon} \ll \frac{X^{f + \varepsilon}}{Y^{2f - 1}} l^{1/2 + \varepsilon}.
\end{equation}

We compute the main term $M_{R1}$ as follows
\begin{multline}
\label{eq:MR1}
M_{R1} = \frac{X}{2 \zeta_2(2) l^{1/2 + \alpha}} \sum_{(c,2l)=1} c^{-2} \lp \sum_{\substack{a |c \\ a > Y}} \mu(a) \rp  \sum_{r | c} \frac{\mu(r)}{r^{1 + 2\alpha}}
\\
\frac{1}{2 \pi i} \int_{(\varepsilon)} X^{\frac{s}{2}} r^{-2s} l^{-s} \frac{G(s)}{s} g_{\alpha}(s)   \widetilde{\Phi}(1 +\tfrac{s}{2})  \zeta_2(1 + 2\alpha + 2s) B_{\alpha+s}(lr) ds.
\end{multline}
Here we freely extended the integration back to $|\text{Im}(s)| \gg \log^2{X}$, again at no cost.

Similarly, we have
\begin{multline}
M_{R2} = \frac{X}{2\zeta_2(2) l^{1/2 - \alpha}} \sum_{(c,2l)=1} c^{-2} \lp \sum_{\substack{a |c \\ a > Y}} \mu(a) \rp  \sum_{r | c} \frac{\mu(r)}{r^{}} 
\\
\frac{1}{2 \pi i} \int_{(\varepsilon)}  X^{\frac{s}{2}}  l^{s} \frac{G(s)}{s} g_{\alpha}(s) \widetilde{\Phi}(1 - \alpha - \tfrac{s}{2}) \zeta_2(1-2\alpha-2s) B_{-\alpha-s}(lr) \leg{X}{c^2}^{-\alpha-s} \gamma_{\alpha+s}  ds.
\end{multline}
That is,
\begin{multline}
\label{eq:MR2}
M_{R2} = \frac{X^{1-\alpha}}{2\zeta_2(2) l^{1/2 - \alpha}} \sum_{(c,2l)=1} c^{-2} \lp \sum_{\substack{a |c \\ a > Y}} \mu(a) \rp  \sum_{r | c} \frac{\mu(r)}{r^{}} 
\\
\frac{1}{2 \pi i} \int_{(\varepsilon)}  X^{-\frac{s}{2}}  l^{s} c^{2\alpha +2s} \frac{G(s)}{s} g_{\alpha}(s) \widetilde{\Phi}(1 - \alpha - \tfrac{s}{2}) \zeta_2(1-2\alpha-2s) B_{-\alpha-s}(lr) \gamma_{\alpha+s}  ds.
\end{multline}

\subsection{Simplifying the Dirichlet series}
At this point we shall simplify the Dirichlet series appearing in the above expressions for $M_{R1}$ and $M_{R2}$.  We need to compute
\begin{equation}
C(w,\alpha+s) := \sum_{(c,2l)=1} \frac{1}{c^{2+w-\alpha-s}} \sum_{\substack{a | c \\ a > Y}} \mu(a) \sum_{r | c} \frac{\mu(r)}{r^{1 + \alpha + s + w}} \zeta_2(1 + 2w) B_w(lr),
\end{equation}
where $w = \pm (\alpha + s)$.
Now we compute
\begin{align}
C(w,\alpha+s) &= \sum_{(c,2l)=1} \frac{1}{c^{2+w-\alpha-s}} \sum_{\substack{a | c \\ a > Y}} \mu(a) \sum_{r | c} \frac{\mu(r)}{r^{1 + \alpha + s + w}} \sum_{(n,2)=1} \frac{1}{n^{1 + 2w}} \prod_{p | lrn} (1+p^{-1})^{-1}
\\
&= \sum_{\substack{a | c \\ a > Y}} \mu(a) \sum_{r} \frac{\mu(r)}{r^{1 + \alpha + s + w}} \sum_{\substack{(c,2l)=1 \\ c \equiv 0 \shortmod{a} \\ c \equiv 0 \shortmod{r}}} \frac{1}{c^{2+w-\alpha-s}}   \sum_{(n,2)=1} \frac{1}{n^{1 + 2w}} \prod_{p | lrn} (1+p^{-1})^{-1},
\end{align}
which simplifies to
\begin{multline}
C(w,\alpha+s) = \zeta_{2l}(2+w-\alpha-s) \prod_{p | l} (1+p^{-1})^{-1} \sum_{\substack{(a,2l)=1\\ a > Y}} \frac{\mu(a)}{a^{2+w-\alpha-s}} 
\\
\sum_{(r,2l)=1} \frac{\mu(r)}{r^{1 + \alpha + s + w}} \leg{(a,r)}{r}^{2 +w-\alpha-s}
\sum_{(n,2)=1} \frac{1}{n^{1 + 2w}} \prod_{p |nr, p \nmid l} (1+p^{-1})^{-1}.
\end{multline}
Thus we get
\begin{multline}
C(w,\alpha+s) = \zeta_{2l}(2+w-\alpha-s) \prod_{p | l} (1+p^{-1})^{-1} (1-p^{-1-2w})^{-1} \sum_{\substack{(a,2l)=1\\ a > Y}} \frac{\mu(a)}{a^{2+w-\alpha-s}} 
\\
\sum_{(r,2l)=1} \frac{\mu(r)}{r^{1 + \alpha + s + w}}  \leg{(a,r)}{r}^{2 +w-\alpha-s}
\sum_{(n,2l)=1} \frac{1}{n^{1 + 2w}} \prod_{p |nr, p \nmid l} (1+p^{-1})^{-1}.
\end{multline}
Now we compute using joint multiplicativity that
\begin{multline}
\sum_{(r,2l)=1} \sum_{(n,2l)=1}  \frac{\mu(r)}{r^{1 + \alpha + s + w}} \leg{(a,r)}{r}^{2 +w-\alpha-s} \frac{1}{n^{1 + 2w}} \prod_{p |nr} (1+p^{-1})^{-1}
\\
=
\prod_{p \nmid 2l} \sum_{0 \leq r \leq 1} \sum_{0 \leq j < \infty} \frac{(-1)^r}{p^{r(1+\alpha+s+w)}}  \leg{(a,p^r)}{p^r}^{2 +w-\alpha-s} \frac{1}{p^{j(1 + 2 w)}} \left[ (1 + p^{-1})^{-1} \right]_{j + r > 0},
\end{multline}
where here $[f(x)]_* = f(x)$ if $*$ is true, and $=1$ otherwise.  It simplifies to
\begin{equation}
\prod_{p \nmid 2l} \left( \sum_{0 \leq j < \infty} \frac{\left[ (1 + p^{-1})^{-1} \right]_{j > 0}}{p^{j(1 + 2 w)}}   - p^{-1-\alpha-s-w}  \leg{(a,p)}{p}^{2 +w-\alpha-s} \sum_{0 \leq j < \infty} \frac{ (1 + p^{-1})^{-1}}{p^{j(1 + 2 w)}} \right),
\end{equation}
which in turn is
\begin{equation}
\prod_{p \nmid 2l} \left( 1 + \frac{p^{-1-2w}}{1-p^{-1-2w}} (1+p^{-1})^{-1} -  \frac{p^{-1-\alpha-s-w}}{1-p^{-1-2w}}  \leg{(a,p)}{p}^{2 +w-\alpha-s} (1 + p^{-1})^{-1}  \right).
\end{equation}
We can simplify this as
\begin{equation}
\prod_{p \nmid 2l} (1-p^{-1-2w})^{-1} \left( 1 -p ^{-1-2w}+ \frac{p^{-1-2w}}{1+p^{-1}} -  \frac{p^{-1-\alpha-s-w}}{1 + p^{-1}} \leg{(a,p)}{p}^{2 +w-\alpha-s}   \right).
\end{equation}

Thus we have
\begin{multline}
C(w,\alpha +s) = \zeta_{2l}(2+w-\alpha-s) \zeta_2(1 + 2w) \prod_{p | l} (1+p^{-1})^{-1} \sum_{\substack{(a,2l)=1\\ a > Y}} \frac{\mu(a)}{a^{2+w-\alpha-s}}
\\
\prod_{p \nmid 2l} \left( 1 -p ^{-1-2w}+ p^{-1-2w} (1+p^{-1})^{-1} -  p^{-1-\alpha-s-w} \leg{(a,p)}{p}^{2 +w-\alpha-s} (1 + p^{-1})^{-1}  \right).
\end{multline}

Simplifying a little bit, we get
\begin{multline}
C(\alpha + s,\alpha +s) = \zeta_{2l}(2) \zeta_2(1 + 2\alpha + 2s) \prod_{p | l} (1+p^{-1})^{-1} \sum_{\substack{(a,2l)=1\\ a > Y}} \frac{\mu(a)}{a^2}
\\
\prod_{p \nmid 2l} \left( 1 -p ^{-1-2\alpha -2s}+ p^{-1-2\alpha -2s} (1+p^{-1})^{-1} -  p^{-1-2\alpha-2s} \frac{(a,p)^2}{p^2} (1 + p^{-1})^{-1}  \right).
\end{multline}
Note
\begin{multline}
1 -p ^{-1-2\alpha -2s}+ p^{-1-2\alpha -2s} (1+p^{-1})^{-1} -  p^{-1-2\alpha-2s} \frac{(a,p)^2}{p^2} (1 + p^{-1})^{-1}
\\
=
\begin{cases}
1 - p^{-1-2\alpha-2s}, \qquad \text{ if } p | a, \\
1 - p^{-2-2\alpha-2s}, \qquad \text{ if } p \nmid a.
\end{cases}
\end{multline}
Furthermore note
\begin{equation}
\zeta_{2l}(2) \prod_{p |l} (1 + p^{-1})^{-1} = \zeta_2(2) \frac{\phi(l)}{l}.
\end{equation}
Thus we have
\begin{equation}
C(\alpha+s,\alpha+s) = \zeta_2(2) \frac{\phi(l)}{l} \sum_{\substack{(a,2l)=1\\ a > Y}} \frac{\mu(a)}{a^2} \frac{\zeta_{2a}(1 + 2\alpha + 2s)}{\zeta_{2al}(2 + 2\alpha + 2s)},
\end{equation}
and hence
\begin{equation}
\label{eq:MR1b}
M_{R1} = \frac{X \phi(l)/l}{2 l^{1/2 + \alpha}}  \sum_{\substack{(a,2l)=1\\ a > Y}} \frac{\mu(a)}{a^2}
\frac{1}{2 \pi i} \int_{(\varepsilon)} X^{\frac{s}{2}} l^{-s} \frac{G(s)}{s} g_{\alpha}(s)   \widetilde{\Phi}(1 +\tfrac{s}{2})  \frac{\zeta_{2a}(1 + 2\alpha + 2s)}{\zeta_{2al}(2 + 2\alpha + 2s)}ds.
\end{equation}

Similarly, we require $C(-\alpha-s, \alpha + s)$, which is
\begin{multline}
C(-\alpha-s,\alpha +s) = \zeta_{2l}(2-2\alpha-2s) \zeta_2(1 -2\alpha-2s) \prod_{p | l} (1+p^{-1})^{-1} \sum_{\substack{(a,2l)=1\\ a > Y}} \frac{\mu(a)}{a^{2-2\alpha-2s}}
\\
\prod_{p \nmid 2l} \left( 1 -p ^{-1+2\alpha +2s}+ p^{-1+2\alpha+2s} (1+p^{-1})^{-1} -  p^{-1} \leg{(a,p)}{p}^{2 -2\alpha-2s} (1 + p^{-1})^{-1}  \right).
\end{multline}
For $p \nmid a$ we get
\begin{multline}
 1 -p ^{-1+2\alpha +2s}+ p^{-1+2\alpha+2s} (1+p^{-1})^{-1} -  p^{-1} \leg{(a,p)}{p}^{2 -2\alpha-2s} (1 + p^{-1})^{-1}
\\
= 1-p^{-2+2\alpha+2s}.
\end{multline}
For $p | a$ it is
\begin{multline}
 1 -p ^{-1+2\alpha +2s}+ p^{-1+2\alpha+2s} (1+p^{-1})^{-1} -  p^{-1}  (1 + p^{-1})^{-1} 
\\
= (1+p^{-1})^{-1}(1-p^{-2+2\alpha+2s}).
\end{multline}
Thus we have
\begin{multline}
M_{R2} = \frac{X^{1-\alpha}}{2\zeta_2(2) l^{1/2 - \alpha}} \prod_{p|l} (1+p^{-1})^{-1} \sum_{\substack{(a,2l)=1 \\ a > Y}} \frac{\mu(a)}{a^2} 
\frac{1}{2 \pi i} \int_{(\varepsilon)}  X^{-\frac{s}{2}}  l^{s} a^{2\alpha +2s} \frac{G(s)}{s} g_{\alpha}(s) \widetilde{\Phi}(1 - \alpha - \tfrac{s}{2}) 
\\
\zeta_2(1-2\alpha-2s) \zeta_{2l}(2-2\alpha-2s) \gamma_{\alpha+s}  
\prod_{p \nmid 2al} (1-p^{-2+2\alpha+2s}) \prod_{p | a} (1+p^{-1})^{-1} (1-p^{-2+2\alpha+2s})
ds,
\end{multline}
which simplifies to
\begin{multline}
\label{eq:MR2b}
M_{R2} = \frac{X^{1-\alpha}}{2\zeta_2(2) l^{1/2 - \alpha}} \prod_{p|l} (1+p^{-1})^{-1} \sum_{\substack{(a,2l)=1 \\ a > Y}} \frac{\mu(a)}{a^2} \prod_{p | a} (1+p^{-1})^{-1}
\\
\frac{1}{2 \pi i} \int_{(\varepsilon)}  X^{-\frac{s}{2}}  l^{s} a^{2\alpha +2s} \frac{G(s)}{s} g_{\alpha}(s) \widetilde{\Phi}(1 - \alpha - \tfrac{s}{2}) 
\zeta_2(1-2\alpha-2s)  \gamma_{\alpha+s}   
ds.
\end{multline}
In summary, we have shown
\begin{myprop}
\label{prop:MR}
Suppose that Conjecture \ref{conj:twistedconj} holds but with an error of size $O(l^{1/2 + \varepsilon} X^{f + \varepsilon})$.  Then
\begin{equation}
 M_R = M_{R1} + M_{R2} +  O\lp\frac{X^{f + \varepsilon}}{Y^{2f-1}} l^{1/2 + \varepsilon}\rp,
\end{equation}
where $M_{R1}$ and $M_{R2}$ are given by \eqref{eq:MR1b} and \eqref{eq:MR2b}.  
\end{myprop}

\subsection{Computing $M_{-R}$}
We can get $M_{-R}$ from Proposition \ref{prop:MR} using Remark \ref{remark:M1toM2}.  Here we compute $M_{-R2}$.
We get
\begin{multline}
M_{-R2} = \frac{X^{}}{2\zeta_2(2) l^{1/2 + \alpha}} \prod_{p|l} (1+p^{-1})^{-1} \sum_{\substack{(a,2l)=1 \\ a > Y}} \frac{\mu(a)}{a^2} \prod_{p | a} (1+p^{-1})^{-1}
\\
\frac{1}{2 \pi i} \int_{(\varepsilon)}  X^{-\frac{s}{2}}  l^{s} a^{-2\alpha +2s} \frac{G(s)}{s} g_{-\alpha}(s) \widetilde{\Phi}(1 - \tfrac{s}{2}) 
\zeta_2(1+2\alpha-2s)  \gamma_{-\alpha+s} \gamma_{\alpha}
ds.
\end{multline}
Now apply the change of variables $s \rightarrow -s$ to get
\begin{multline}
M_{-R2} = -\frac{X}{2\zeta_2(2) l^{1/2 + \alpha}} \prod_{p|l} (1+p^{-1})^{-1} \sum_{\substack{(a,2l)=1 \\ a > Y}} \frac{\mu(a)}{a^2} \prod_{p | a} (1+p^{-1})^{-1}
\\
\frac{1}{2 \pi i} \int_{(-\varepsilon)}  X^{\frac{s}{2}}  l^{-s} a^{-2\alpha -2s} \frac{G(s)}{s} g_{-\alpha}(-s) \widetilde{\Phi}(1 + \tfrac{s}{2}) 
\zeta_2(1+2\alpha+2s)  \gamma_{-\alpha-s} \gamma_{\alpha}
ds.
\end{multline}
Next notice that
\begin{equation}
 g_{-\alpha}(-s)\gamma_{-\alpha-s} \gamma_{\alpha} = g_{\alpha}(s),
\end{equation}
so we get
\begin{multline}
\label{eq:M-R2b}
M_{-R2} = -\frac{X}{2\zeta_2(2) l^{1/2}} \prod_{p|l} (1+p^{-1})^{-1} \sum_{\substack{(a,2l)=1 \\ a > Y}} \frac{\mu(a)}{a^2} \prod_{p | a} (1+p^{-1})^{-1}
\\
\frac{1}{2 \pi i} \int_{(-\varepsilon)}  X^{\frac{s}{2}}  l^{-\alpha-s} a^{-2\alpha -2s} \frac{G(s)}{s} g_{\alpha}(s) \widetilde{\Phi}(1 + \tfrac{s}{2}) 
\zeta_2(1+2\alpha+2s) 
ds.
\end{multline}

\section{Computing $M_N$}
\label{section:MN}
Recall
\begin{equation}
M_N =  \sum_{\substack{(a,2l) = 1 \\ a \leq Y}} \mu(a)  \sum_{(n,2a)=1} \frac{\leg{8}{ln} }{n^{\thalf + \alpha}} \sum_{(d,2)=1} \leg{d}{ln} \Phi\left(\frac{d a^2}{X}\right) V_{\alpha} \left(\frac{n}{a \sqrt{d}}\right).
\end{equation}
\subsection{Application of Poisson summation}
The Poisson summation formula (Proposition \ref{lemma:Poisson}) shows
\begin{multline}
 \sum_{(d,2)=1} \leg{d}{ln} \Phi\left(\frac{d a^2}{X}\right) V_{\alpha} \left(\frac{n}{a \sqrt{d}}\right) 
\\
= \frac{1}{2ln} \leg{2}{ln} \sum_{k \in \mz} (-1)^k G_k(ln) \intR (C+S)\leg{2 \pi kx}{2ln} \Phi\leg{xa^2}{X} V_{\alpha}\leg{n}{ a \sqrt{x}} dx.
\end{multline}
Here $(C+S)(t)$ means $\cos(t) + \sin(t)$.  Thus we have
\begin{multline}
\label{eq:MNPoisson}
 M_N = \frac{X}{2} \sum_{\substack{(a,2l) = 1 \\ a \leq Y}} \frac{\mu(a)}{a^2}  \sum_{(n,2a)=1} \frac{1}{n^{\thalf + \alpha}}  \sum_{k \in \mz} (-1)^k \frac{G_k(ln)}{ln} 
\\
\intR (C+S)\leg{2 \pi kxX}{2lna^2} \Phi(x) V_{\alpha}\leg{n}{\sqrt{xX}} dx.
\end{multline}

Now write $M_N = M_N(k=0) + M_N(k \neq 0)$, where naturally $M_N(k=0)$ corresponds to the term with $k=0$.  

\subsection{Computation of $M_N(k=0)$}
By Lemma \ref{lemma:Gk} we have $G_0(ln) = \phi(ln)$ if $ln = \square$ (i.e. $n = l \square$, since $l$ is odd and squarefree), and $0$ otherwise.  Thus we get
\begin{equation}
 M_N(k=0) = \frac{X}{2} l^{-1/2 - \alpha} \sum_{\substack{(a,2l) = 1 \\ a \leq Y}} \frac{\mu(a)}{a^2}  \sum_{(n,2a)=1} \frac{1}{n^{1 + 2\alpha}} \frac{\phi(ln)}{ln} \int_0^{\infty} \Phi(x) V_{\alpha}\leg{ln^2}{\sqrt{xX}} dx.
\end{equation}
Now we compute the relevant Dirichlet series as follows
\begin{align}
\sum_{(n,2a)=1} \frac{1}{n^{1 + 2\alpha + 2s}} \frac{\phi(ln)}{ln} 
 &=
\frac{\phi(l)}{l} \prod_{p \nmid 2al} \lp 1 + (1-p^{-1}) \frac{p^{-1-2\alpha-2s}}{1-p^{-1-2\alpha-2s}} \rp 
\prod_{p \nmid 2a, p | l} \lp1 - p^{-1-2\alpha-2s}\rp^{-1}
\\
 &= \frac{\phi(l)}{l}   \prod_{p \nmid 2al} \frac{(1 - p^{-2 -2\alpha - 2s})}{(1 - p^{-1 - 2\alpha -2s})} \prod_{p \nmid 2a, p | l} \lp1 - p^{-1-2\alpha-2s}\rp^{-1},
\end{align}
which simplifies to
\begin{equation}
\frac{\phi(l)}{l} \zeta_{2a}(1 + 2\alpha + 2s) \prod_{p \nmid 2al} (1 - p^{-2 -2\alpha - 2s})
= \frac{\phi(l)}{l} \frac{\zeta_{2a}(1 + 2\alpha + 2s)}{\zeta_{2al}(2 + 2\alpha + 2s)}.
\end{equation}
Using the integral representation of $V$ (i.e. \eqref{eq:V}), and the Mellin transform of $\Phi$, we get
\begin{multline}
\label{eq:P1}
  M_N(k=0) = \frac{X \phi(l)/l}{2 l^{1/2 + \alpha}} \sum_{\substack{(a,2l) = 1 \\ a \leq Y}} \frac{\mu(a)}{a^2}    
\\
\frac{1}{2\pi i} \int_{(\varepsilon)} \widetilde{\Phi}(1 + \tfrac{s}{2})\frac{G(s)}{s} g_{\alpha}(s) X^{\frac{s}{2}} l^{-s} \frac{\zeta_{2a}(1 + 2\alpha + 2s)}{\zeta_{2al}(2 + 2\alpha + 2s)} ds.
\end{multline}

\subsection{Extracting a secondary term from $M_N(k \neq 0)$}
It turns out, but is not at all obvious, that $M_N(k = 0)$ and $M_{-N}(k \neq 0)$ combine naturally.  The reason for this simplification has to do with the sum over $k$'s that are squares.  In his work on the mollified second moment (and the third moment) of this family, Soundararajan extracted certain main terms from such square $k$'s \cite{Sound}.  The situation here is somewhat different, in that both $M_{N}(k=0)$ and the sum over square $k$'s seem to be individually intractible beyond an error term of $O(X^{3/4 + \varepsilon})$, at least with our present knowledge of the zero-free region of the Riemann zeta function (here by intractible we simply mean that the usual procedure of moving contours of integration leads naturally to an error of $O(X^{3/4 + \varepsilon})$, and the presence of a factor $\zeta^{-1}$ prevents moving the contour any further).

Recall 
\begin{multline}
 M_N(k \neq 0) = \frac{X}{2} \sum_{\substack{(a,2l) = 1 \\ a \leq Y}} \frac{\mu(a)}{a^2}  \sum_{(n,2a)=1} \frac{1}{n^{\thalf + \alpha}}  \sum_{k \neq 0} (-1)^k \frac{G_k(ln)}{ln} 
\\
\intR (C+S)\leg{2 \pi kxX}{2lna^2} \Phi(x) V_{\alpha}\leg{n}{\sqrt{xX}} dx.
\end{multline}

\subsection{Manipulations with the test functions}
In this section we replace the previous Fourier-type integral transform with a Mellin-type integral.  By using the Mellin transforms of $\Phi$ and $V_{\alpha}$, we get
\begin{multline}
\label{eq:fourierintegral}
\intR  (C+S) \left(\frac{2 \pi x kX}{2lna^2} \right)  \Phi(x) V_{\alpha} \left(\frac{n}{(xX)^{1/2}}\right) dx 
\\
= \int_0^{\infty} (C+S)\left(\frac{2 \pi x kX}{2lna^2} \right)  \leg{1}{2\pi i}^2 \int_{(c_u)} \int_{(c_s)} \widetilde{\Phi}(u) x^{-u}  \frac{G(s)}{s} g_{\alpha}(s) n^{-s} (xX)^{\frac{s}{2}} ds du dx,
\end{multline}
which after a change of variables becomes
\begin{equation}
\int_0^{\infty} (C+S)\left( \text{sgn}(k) x \right)  \leg{1}{2\pi i}^2 \int_{(c_u)} \int_{(c_s)} \widetilde{\Phi}(1+u) \frac{X^{u}}{n^s} \leg{lna^2 x}{\pi |k| }^{-u + \tfrac{s}{2}}  \frac{G(s)}{s} g_{\alpha}(s)  ds du \frac{dx}{x}.
\end{equation}

From \cite{GR}, 17.43.3, 17.43.4, we have (where $CS$ stands for $\cos$ or $\sin$), 
\begin{equation}
 \int_0^{\infty} CS(x) x^{w} \frac{dx}{x} = \Gamma(w) CS(\tfrac{\pi w}{2}),
\end{equation}
for $0 < \text{Re}(w) < 1$.  By interchanging the orders of integration (justified, say, by dissecting the $x$-integral into $x \leq A$ and $x > A$; one treats the $x > A$ integral by a contour shift, while the $x \leq A$ integral is interchanged with the $s$ and $u$ integrals, and then extended to all $x$ by integration by parts), we obtain that \eqref{eq:fourierintegral} is
\begin{multline}
\leg{1}{2\pi i}^2 \int_{(c_u)} \int_{(c_s)} \widetilde{\Phi}(1+u) X^{u} \leg{lna^2}{\pi |k| }^{\tfrac{s}{2}-u}  \frac{G(s)}{s} g_{\alpha}(s) n^{-s} 
\\
\Gamma(\tfrac{s}{2}-u) (C + \text{sgn}(k) S)\lp\tfrac{\pi}{2}(\tfrac{s}{2}-u)\rp ds du.
\end{multline}
Here the requirement is $0 < \thalf c_s - c_u < 1$.  We initially take $c_u = 0$, $c_s = \varepsilon$.  Consequently, we get
\begin{multline}
\label{eq:MNpreDirichlet}
 M_N(k \neq 0) = \frac{X}{2} \sum_{\substack{(a,2l) = 1 \\ a \leq Y}} \frac{\mu(a)}{a^2}  \sum_{(n,2a)=1} \frac{1}{n^{\thalf + \alpha}}  \sum_{k \neq 0} (-1)^k \frac{G_k(ln)}{ln} 
\leg{1}{2\pi i}^2 
\\
\int_{(c_u)} \int_{(c_s)} 
\widetilde{\Phi}(1+u) X^{u} \leg{lna^2}{\pi |k| }^{\tfrac{s}{2}-u}  \frac{G(s)}{s} g(s) n^{-s} 
\Gamma(\tfrac{s}{2}-u) (C + \text{sgn}(k) S)\lp\tfrac{\pi}{2}(\tfrac{s}{2}-u)\rp ds du.
\end{multline}
Next we take $c_s$ sufficiently large so that the sums over $n$ and $k$ converge absolutely.

\subsection{A Dirichlet series computation}
We need to compute
\begin{equation}
H(k_1,l;v,w) = \sum_{k_2=1}^{\infty} \sum_{(n,2a) = 1} \frac{G_{k_1 k_2^2}(ln)/ln}{k_2^v n^{w}},
\end{equation}
where $k_1$ is squarefree, and recall that $l$ is odd and squarefree.  It is clear from Lemma \ref{lemma:Gk} that $G_{k_1 k_2^2}(n)$ is multiplicative in $k_2$ and $n$.  The analogous computation over squares here was necessary for Soundararajan to obtain the full main term in the mollified second moment of this family \cite{Sound}, while here we require it to cancel certain other terms; it does not give a main term.

In this section, we show
\begin{mylemma}
 We have
\begin{equation}
 H(k_1,l;v,w) =  \frac{\zeta_l(v) \zeta_{2al}(v+2w) L_{2al}(\thalf + w, \chi_{k_1})}{\zeta_{2al}(1+2w) L_{2al}(\thalf + v + w, \chi_{k_1})} \prod_{p |l} J_p(k_1;v,w),
\end{equation}
where
\begin{equation}
 J_p(k_1;v,w) =  p^{w}  \frac{ -(1-p^{-v-2w}) \lp 1- \leg{k_1}{p} p^{-\half - w} \rp + (1-p^{-1-2w})  \lp1- \leg{k_1}{p} p^{-\half-v-w} \rp }{(1-p^{-v}) (1-p^{-v-2w}) \lp 1- \leg{k_1}{p} p^{-\half - w} \rp},
\end{equation}
and where the meaning of the subscript $l$ on $\zeta_{l}$, etc. is that the Euler factors at primes $p$ with $p | l$ are removed.
Furthermore,
\begin{equation}
\label{eq:Jestimate}
 J_p(k_1;v,w) \ll p^{-1/2},
\end{equation}
for $\text{Re}(v) \geq 2$, $\text{Re}(w) \geq 0$.
\end{mylemma}

\begin{proof}
To begin, we factor $H$ as follows
\begin{align}
H(k_1,l;v,w) &= \prod_{p | 2a} (1 - p^{-v})^{-1} \sum_{(k_2,2a)=1}\sum_{(n,2a) = 1} \frac{G_{k_1 k_2^2}(ln)/ln}{k_2^v n^{w}}
\\
&= \prod_{p | 2a} (1 - p^{-v})^{-1} \lp \sum_{(k_2,2al)=1}\sum_{(n,2al) = 1} \frac{G_{k_1 k_2^2}(n)/n}{k_2^v n^{w}} \rp \lp \prod_{k_2, n | l^{\infty}} \frac{G_{k_1 k_2^2}(ln)/ln}{k_2^v n^{w}} \rp.
\end{align}
Now we need to compute
\begin{equation}
H_p(k_1;v,w) = \sum_{j=0}^{\infty} \sum_{r=0}^{\infty} \frac{G_{k_1 p^{2j}}(p^r)/p^r}{p^{jv +rw}},
\end{equation}
where there are two cases depending on if $p \nmid k_1$ or $p | k_1$.  For the former case, we get
\begin{equation}
 H_p(k_1;v,w) =  \sum_{0 \leq r \leq j < \infty} \frac{\phi(p^{2r})/p^{2r}}{p^{jv +2rw}} + \leg{k_1}{p} \sum_{j=0}^{\infty} \frac{p^{2j + \half} p^{-2j-1}}{p^{jv + (2j+1)w}},
\end{equation}
which we simplify as
\begin{align}
 H_p(k_1;v,w) &=  (1-p^{-v})^{-1} \sum_{r=0}^{\infty} \frac{\phi(p^{2r})/p^{2r}}{p^{r(2w +v)}} + \leg{k_1}{p} p^{-\half - w} (1-p^{-v-2w})^{-1}
\\
&=(1-p^{-v})^{-1} \lp1 + (1-\frac{1}{p}) \frac{p^{-v-2w}}{1-p^{-v-2w}} \rp + \leg{k_1}{p} p^{-\half - w} (1-p^{-v-2w})^{-1}
\\
&=(1-p^{-v-2w})^{-1} \lp (1-p^{-v})^{-1} \lp 1 - p^{-v-2w} + (1-\frac{1}{p}) p^{-v-2w} \rp  +  \leg{k_1}{p} p^{-\half - w}\rp
\\
&= (1-p^{-v-2w})^{-1} \lp (1-p^{-v})^{-1} \lp 1 - p^{-1-v-2w} \rp  +  \leg{k_1}{p} p^{-\half - w}\rp
\\
&= (1-p^{-v-2w})^{-1} (1-p^{-v})^{-1} \lp  \lp 1 - p^{-1-v-2w} \rp  +  (1-p^{-v}) \leg{k_1}{p} p^{-\half - w}\rp.
\end{align}

Next pull out the factor $(1-(k_1/p) p^{-\half-w})^{-1}$, and note
\begin{align}
 \lp1-\leg{k_1}{p} p^{-\half-w}\rp \lp  \lp 1 - p^{-1-v-2w} \rp  +  (1-p^{-v}) \leg{k_1}{p} p^{-\half - w}\rp
\\
= 1 - p^{-1-v-2w} + (1-p^{-v})\leg{k_1}{p} p^{-\half-w} - \leg{k_1}{p} p^{-\half-w} (1 - p^{-1-v-2w})  - (1-p^{-v}) p^{-1-2w}
\\
= (1 - p^{-1-2w}) \lp1- \leg{k_1}{p} p^{-\half-v-w} \rp. 
\end{align}
In conclusion, we get
\begin{equation}
\label{eq:Hp}
 H_p(k_1;v,w) = \frac{(1-p^{-1-2w})  \lp1- \leg{k_1}{p} p^{-\half-v-w} \rp}{ (1-p^{-v})(1-p^{-v-2w}) \lp 1- \leg{k_1}{p} p^{-\half - w} \rp},
\end{equation}
for $p \nmid k_1$.

The case where $p | k_1$ is similar, and we compute in this situation that
\begin{align}
 H_p(k_1;v,w) &= \sum_{0 \leq r \leq j < \infty} \frac{\phi(p^{2r})/p^{2r}}{p^{jv +2rw}} - \sum_{j=0}^{\infty} \frac{p^{2j + 1} p^{-2j-2}}{p^{jv + (2j+2)w}}
\\
&= (1-p^{-v})^{-1} \lp1 + (1-\frac{1}{p}) \frac{p^{-v-2w}}{1-p^{-v-2w}} \rp - p^{-1-2w} (1-p^{-v-2w})^{-1}
\\
&= (1-p^{-v-2w})^{-1} \lp (1-p^{-v})^{-1} (1-p^{-1-v-2w}) -p^{-1-2w}) \rp
\\
&= (1-p^{-v-2w})^{-1} (1-p^{-v})^{-1} \lp (1-p^{-1-v-2w}) - (1-p^{-v}) p^{-1-2w} \rp
\\
&= (1-p^{-v-2w})^{-1} (1-p^{-v})^{-1} (1-p^{-1-2w}).
\end{align}
Notice that the expression \eqref{eq:Hp} reduces to this one for $p | k_1$, so we may actually use \eqref{eq:Hp} for all $k_1$.

Finally we need to compute the contribution from $n, k |l^{\infty}$ (the notation meaning all the primes dividing $n$ and $k$ also divide $l$).  In this case the computation reduces to
\begin{align}
J_p(k_1;v,w)&:= \sum_{j=0}^{\infty} \sum_{r=0}^{\infty} \frac{G_{k_1 p^{2j}}(p^{r+1})/p^{r+1}}{p^{jv + rw}} 
\\
&= p^{w} \sum_{j=0}^{\infty} \sum_{r=0}^{\infty} \frac{G_{k_1 p^{2j}}(p^{r+1})/p^{r+1}}{p^{jv + (r +1)w}}
=
p^{w} \sum_{j=0}^{\infty} \sum_{r=1}^{\infty} \frac{G_{k_1 p^{2j}}(p^{r})/p^{r}}{p^{jv + rw}}
\\
&= p^{w} \lp -\frac{1}{1-p^{-v}} + \frac{(1-p^{-1-2w})  \lp1- \leg{k_1}{p} p^{-\half-v-w} \rp}{ (1-p^{-v})(1-p^{-v-2w}) \lp 1- \leg{k_1}{p} p^{-\half - w} \rp} \rp
\\
&=  p^{w}  \frac{ -(1-p^{-v-2w}) \lp 1- \leg{k_1}{p} p^{-\half - w} \rp + (1-p^{-1-2w})  \lp1- \leg{k_1}{p} p^{-\half-v-w} \rp }{(1-p^{-v}) (1-p^{-v-2w}) \lp 1- \leg{k_1}{p} p^{-\half - w} \rp} 
\end{align}
The estimate \eqref{eq:Jestimate} is clear from the previous representation.
\end{proof}

\subsection{Another Dirichlet series computation}
We need to slightly generalize the computation from the previous section.  Namely, we require a formula for
\begin{equation}
A_{\epsilon,l}(u,w):= \sum_{(n,2a)=1}  \sum_{k =1}^{\infty} (-1)^k \frac{G_{\epsilon k}(ln)/ln}{ |k|^u n^{w}},
\end{equation}
where $\epsilon = \pm 1$.
Of course it is also natural to write $k = k_1 k_2^2$ where $k_1$ is squarefree.  Then we sum over $k_1$ odd and $k_1$ even separately.  Thus we get
\begin{align}
 A_{\epsilon,l}(u,w) &= \sumstar_{k_1 >0} \frac{1}{|k_1|^u} \sum_{(n,2a)=1}  \sum_{k_2=1}^{\infty} (-1)^{k_1 k_2} \frac{G_{\epsilon k_1 k_2^2}(ln)/ln}{ |k_2|^{2u} n^{w}}
\\
&=  \sumstar_{k_1 \text{ even}} \frac{1}{|k_1|^u} \sum_{(n,2a)=1}  \sum_{k_2=1}^{\infty} \frac{G_{\epsilon k_1 k_2^2}(ln)}{ln |k_2|^{2u} n^{w}} + \sumstar_{k_1 \text{ odd}} \frac{1}{|k_1|^u} \sum_{(n,2a)=1}  \sum_{k_2=1}^{\infty} (-1)^{k_2} \frac{G_{\epsilon k_1 k_2^2}(ln)}{ln |k_2|^{2u} n^{w}}.
\end{align}

For each fixed $k_1$ we have computed the inner Dirichlet series on the left (it is $H(\epsilon k_1,l;v,w)$).  Then we compute
\begin{equation}
H_{-1}(\epsilon k_1,l;v,w) := \sum_{(n,2a)=1}  \sum_{k_2=1}^{\infty} (-1)^{k_2} \frac{G_{\epsilon k_1 k_2^2}(ln)/ln}{ k_2^{v} n^{w}}
\end{equation}
We first do some simple manipulations to get rid of the $(-1)^{k_2}$ factor, getting
\begin{equation}
H_{-1}(\epsilon k_1,l;v,w) = \sum_{k_2 \text{ even}} \sum_{(n,2a) = 1} \frac{G_{\epsilon k_1 k_2^2}(ln)/ln}{k_2^v n^{w}} - \sum_{k_2 \text{ odd}} \sum_{(n,2a) = 1} \frac{G_{\epsilon k_1 k_2^2}(ln)/ln}{k_2^v n^{w}}.
\end{equation}
To the the latter we use inclusion-exclusion to remove the condition that $k_2$ is odd and then for the sum over $k_2$ even we apply the change of variables $k_2 \rightarrow 2k_2$, using $G_{4 \epsilon k_1 k_2^2}(ln) = G_{\epsilon k_1 k_2^2}(ln)$ since $ln$ is odd, to get
\begin{align}
H_{-1}(\epsilon k_1,l;v,w) &= 2^{-v} \sum_{k_2=1}^{\infty} \sum_{(n,2a) = 1} \frac{G_{\epsilon k_1 k_2^2}(ln)/ln}{k_2^v n^{w}} - (1-2^{-v}) \sum_{k_2 =1}^{\infty} \sum_{(n,2a) = 1} \frac{G_{\epsilon k_1 k_2^2}(ln)/ln}{k_2^v n^{w}}
\\
&= (2^{1-v}-1) H(\epsilon k_1,l;v,w).
\end{align}
Thus we have
\begin{equation}
 A_{l,\epsilon}(u,w) = \sumstar_{k_1 \text{ even}} \frac{1}{|k_1|^u} H(\epsilon k_1,l;2u,w) + \sumstar_{k_1 \text{ odd}} \frac{1}{|k_1|^u} (2^{1-2u}-1) H(\epsilon k_1,l;2u,w).
\end{equation}
Using our computation of $H(\epsilon k_1,l;v,w)$, we get
\begin{multline}
 A_{l,\epsilon}(u,w) = \frac{\zeta_l(2u) \zeta_{2al}(2u+2w)}{\zeta_{2al}(1+2w)} \left(\sumstar_{ k_1 \text{ even}} \frac{1}{|k_1|^u} + \sumstar_{k_1 \text{ odd}} \frac{1}{|k_1|^u} (2^{1-2u}-1) \right) 
\\
\frac{ L_{2al}(\thalf + w, \chi_{\epsilon k_1})}{ L_{2al}(\thalf + 2u + w, \chi_{\epsilon k_1})} \prod_{p |l} J_p(\epsilon k_1;2u,w).
\end{multline}

\subsection{Calculating $M_N(k \neq 0)$}
With these Dirichlet series in hand, we now return to our calculation of $M_N(k \neq 0)$ from \eqref{eq:MNpreDirichlet}, giving
\begin{multline}
 M_N(k \neq 0) = \frac{X}{2} \sum_{\substack{(a,2l) = 1 \\ a \leq Y}} \frac{\mu(a)}{a^2} 
\leg{1}{2\pi i}^2 \int_{(c_u =0)} \int_{(c_s = 2)} \widetilde{\Phi}(1+u) X^{u} \leg{la^2}{\pi }^{\tfrac{s}{2}-u}  \frac{G(s)}{s} g_{\alpha}(s) 
\\
\sum_{\epsilon = \pm 1} A_{l,\epsilon}(\tfrac{s}{2} - u, \thalf + \alpha + \tfrac{s}{2} + u) 
\Gamma(\tfrac{s}{2}-u) (C + \text{sgn}(\epsilon) S)\lp\tfrac{\pi}{2}(\tfrac{s}{2}-u)\rp ds du.
\end{multline}
Note
\begin{multline}
 A_{l,\epsilon}(\tfrac{s}{2} - u, \thalf + \alpha + \tfrac{s}{2} + u)
  = \frac{\zeta_l(s-2u) \zeta_{2al}(1 + 2\alpha + 2s)}{\zeta_{2al}(2+2\alpha + s + 2u)} 
\left(\sumstar_{k_1 \text{ even}} \frac{1}{|k_1|^{\tfrac{s}{2}-u}} + \sumstar_{k_1 \text{ odd}} \frac{(2^{1-s+2u}-1)}{|k_1|^{\tfrac{s}{2}-u}}  \right) 
\\
\frac{ L_{2al}(1 + \alpha + \tfrac{s}{2} +u, \chi_{\epsilon k_1})}{ L_{2al}(1 + \tfrac{3s}{2} + \alpha -u, \chi_{\epsilon k_1})} \prod_{p |l} J_p(\epsilon k_1;s-2u,\thalf + \alpha + \tfrac{s}{2} + u).
\end{multline}
Now we move $c_s$ to $1/2 + \varepsilon$, and $c_u$ to $-3/4$.  It is evident that we crossed a pole at $\alpha + \tfrac{s}{2} + u = 0$ when $k_1 =1$, $\epsilon =1$ only.  On the new line we bound everything trivially and use the known estimate on the first moment of quadratic Dirichlet L-functions to see that the sum over $k_1$ converges absolutely on these lines of integration.

Letting $M_N(k_1 = 1)$ be the contribution from the residue which appears with $k_1 = 1$, $\epsilon=1$ only, we get
\begin{equation}
  M_N(k \neq 0) = M_N(k_1 = 1) + O(X^{1/4 + \varepsilon} Y l^{1/2 + \varepsilon}).
\end{equation}

\subsection{Computing $M_N(k_1 = 1)$}
Note that the residue of $A_{l,1}(\tfrac{s}{2} -u, \thalf + \alpha + \tfrac{s}{2} + u)$ at $u = -\alpha - \tfrac{s}{2}$ is
\begin{equation}
\frac{\phi(2al)}{2al} \frac{\zeta_l(2s + 2\alpha)}{\zeta_{2al}(2)}  (2^{1-2s -2\alpha}-1)  
\prod_{p |l} J_p(1;2s + 2\alpha,\thalf).
\end{equation}
Also note
\begin{align}
 J_p(1;2s + 2\alpha,\thalf) &= p^{1/2}  \frac{ -(1-p^{-1 - 2\alpha -2s}) \lp 1- p^{-1} \rp + (1-p^{-2})  \lp1-  p^{-1-2\alpha -2s} \rp }{(1-p^{-2\alpha -2s}) (1-p^{-1 -2\alpha -2s}) \lp 1-  p^{-1} \rp}
\\
&=
p^{-1/2}  (1-p^{-2\alpha -2s})^{-1}.
\end{align}
Thus we get that the residue of $A_{l,1}$ is
\begin{equation}
 \frac{\phi(a) \phi(l)}{2al} l^{-1/2} \frac{\zeta(2s + 2\alpha)}{\zeta_{2al}(2)}  (2^{1-2s -2\alpha}-1).
\end{equation}

Inserting this above, we get
\begin{multline}
 M_N(k_1 = 1) = \frac{X}{2} \sum_{\substack{(a,2l) = 1 \\ a \leq Y}} \frac{\mu(a)}{a^2} 
\frac{1}{2\pi i}\int_{(\varepsilon)} \widetilde{\Phi}(1-\alpha - \tfrac{s}{2}) X^{-\alpha - \frac{s}{2}} \leg{la^2}{\pi }^{\alpha+ s}  \frac{G(s)}{s} g_{\alpha}(s) 
\\
 \frac{\phi(a) \phi(l)}{2al} l^{-1/2} \frac{\zeta(2\alpha +2s)}{\zeta_{2al}(2)}  (2^{1-2\alpha -2s}-1)
\Gamma(\alpha + s) (C + S)\lp\tfrac{\pi}{2}(\alpha + s)\rp ds.
\end{multline}

\subsection{Simplifying $M_N(k_1 = 1)$}
In this section we show
\begin{myprop}
 We have
\begin{multline}
 M_N(k_1 = 1) = -\frac{X^{1-\alpha} \phi(l)/l}{2 l^{1/2}}  \gamma_{\alpha} \sum_{\substack{(a,2l) = 1 \\ a \leq Y}} \frac{\mu(a) \phi(a)/a}{a^2 \zeta_{2al}(2)} 
\frac{1}{2\pi i}\int_{(-\varepsilon)} \widetilde{\Phi}(1-\alpha + \tfrac{s}{2}) X^{  \frac{s}{2}} (la^2)^{ \alpha - s} \\ 
\frac{G(s)}{s} g_{-\alpha}(s)
 \zeta_2(1-2\alpha +2s) 
 ds,
\end{multline}
and
\begin{multline}
\label{eq:M-Nk1=1}
 M_{-N}(k_1 = 1) = -\frac{X }{2 \zeta_2(2) l^{1/2}} \prod_{p|l} (1+p^{-1})^{-1}   \sum_{\substack{(a,2l) = 1 \\ a \leq Y}} \frac{\mu(a) }{a^2} \prod_{p|a}(1+p^{-1})^{-1}
\\
\frac{1}{2\pi i}\int_{(-\varepsilon)} \widetilde{\Phi}(1 + \tfrac{s}{2}) X^{  \frac{s}{2}} (la^2)^{- \alpha -s}  
\frac{G(s)}{s} g_{\alpha}(s)
 \zeta_2(1+2\alpha + 2s) 
 ds.
\end{multline}
\end{myprop}

\begin{proof}
First use the functional equation for the Riemann zeta function in the form
\begin{equation}
 \pi^{-\alpha -s} \Gamma(\alpha + s) \zeta(2\alpha +2s) = \pi^{-\thalf + \alpha + s} \Gamma(\thalf -\alpha -s) \zeta(1-2\alpha -2s),
\end{equation}
giving
\begin{multline}
 M_N(k_1 = 1) = \frac{X \phi(l)/l}{2 l^{1/2}} \half \sum_{\substack{(a,2l) = 1 \\ a \leq Y}} \frac{\mu(a) \phi(a)/a}{a^2 \zeta_{2al}(2)} 
\frac{1}{2\pi i}\int_{(\varepsilon)} \widetilde{\Phi}(1-\alpha - \tfrac{s}{2}) X^{-\alpha - \frac{s}{2}} (la^2)^{\alpha + s}  \frac{G(s)}{s} g_{\alpha}(s) 
\\
\pi^{-\thalf + \alpha + s} \Gamma(\thalf -\alpha -s)  \zeta(1-2\alpha -2s)  (2^{1 -2\alpha -2s}-1) (C + S)\lp\tfrac{\pi}{2}(\alpha +s)\rp ds.
\end{multline}
Next note
\begin{equation}
 (2^{1-2\alpha -2s} -1) \zeta(1-2\alpha -2s) = 2^{1-2\alpha -2s} \zeta_2(1-2\alpha -2s),
\end{equation}
so we get
\begin{multline}
 M_N(k_1 = 1) = \frac{X^{1-\alpha} \phi(l)/l}{2 l^{1/2}}  \sum_{\substack{(a,2l) = 1 \\ a \leq Y}} \frac{\mu(a) \phi(a)/a}{a^2 \zeta_{2al}(2)} 
\frac{1}{2\pi i}\int_{(\varepsilon)} \widetilde{\Phi}(1-\alpha - \tfrac{s}{2}) X^{-\frac{s}{2}} (la^2)^{\alpha + s}  \frac{G(s)}{s} 
\\
\half \pi^{-\thalf + \alpha + s} g_{\alpha}(s) \Gamma(\thalf -\alpha -s)    2^{1 -2\alpha -2s} (C + S)\lp\tfrac{\pi}{2}(\alpha + s)\rp \zeta_2(1-2\alpha -2s)ds.
\end{multline}
Next we manipulate the gamma functions and $C+ S$ term, namely
\begin{multline}
\half g_{\alpha}(s) \pi^{-\thalf + \alpha +s} \Gamma(\thalf -\alpha -s)  2^{1 -2\alpha -2s} (C + S)\lp\tfrac{\pi}{2}(\alpha + s)\rp 
\\
=  \frac{\Gamma\leg{\thalf + \alpha + s}{2}}{\Gamma\leg{\thalf + \alpha}{2}}
 \pi^{-\thalf + \alpha  + \tfrac{s}{2}} \Gamma(\thalf  -\alpha -s)  2^{ -2\alpha - \frac{s}{2}} (C + S)\lp\tfrac{\pi}{2}(\alpha + s )\rp.
\end{multline}
Now we use the trigonometric identity
\begin{equation}
 \cos(\theta) + \sin(\theta) = \sqrt{2} \cos(\tfrac{\pi}{4} -\theta)
\end{equation}
to get
\begin{align}
\half g_{\alpha}(s) \pi^{-\thalf +\alpha + s} \Gamma(\thalf -\alpha -s )  2^{1 -2\alpha -2s} (C + S)\lp\tfrac{\pi}{2}(\alpha + s)\rp 
\\
=  \frac{\Gamma\leg{\thalf + \alpha + s}{2}}{\Gamma\leg{\thalf + \alpha}{2}}
 \pi^{-\thalf+ \alpha  + \tfrac{s}{2} } \Gamma(\thalf -\alpha  -s)  2^{\thalf -2\alpha  - \frac{s}{2}} \cos\lp\tfrac{\pi}{2}(\thalf- \alpha -s )\rp.
\end{align}
Next we use gamma function identity (see chapter 10 of \cite{Davenport})
\begin{equation}
\pi^{-\thalf} 2^{1-u} \cos(\tfrac{\pi}{2} u) \Gamma(u) = \frac{\Gamma \leg{u}{2}}{\Gamma\leg{1-u}{2}}
\end{equation}
for $u=\half - \alpha -s$ to get
\begin{multline}
 \frac{\Gamma\leg{\thalf + \alpha + s}{2}}{\Gamma\leg{\thalf + \alpha}{2}} 
 \pi^{-\thalf + \alpha  + \tfrac{s}{2}} \Gamma(\thalf -\alpha -s)  2^{\thalf  -2\alpha - \frac{s}{2}} \cos\lp\tfrac{\pi}{2}(\thalf - \alpha -s)\rp
\\
 =  \frac{\Gamma\leg{\thalf + \alpha + s}{2}}{\Gamma\leg{\thalf + \alpha}{2}}
 \pi^{-\thalf  + \alpha + \tfrac{s}{2}}   2^{\thalf  -2\alpha - \frac{s}{2}} 
\pi^{\thalf} 2^{-1 + (\half - \alpha -s)}
\frac{\Gamma \leg{\half - \alpha - s }{2}}{\Gamma\leg{1-(\half - \alpha  - s)}{2}}
= 
\frac{\Gamma \leg{\half - \alpha - s }{2}}{\Gamma\leg{\thalf + \alpha}{2}}
 \leg{8}{\pi}^{-\alpha-\frac{s}{2}}.
\end{multline}
Thus we get
\begin{equation}
 M_N(k_1 = 1) = \frac{X^{1-\alpha} \phi(l)/l}{2 l^{1/2}} \sum_{\substack{(a,2l) = 1 \\ a \leq Y}} \frac{\mu(a) \phi(a)/a}{a^2 \zeta_{2al}(2)} I,
\end{equation}
where
\begin{equation}
 I = \frac{1}{2\pi i}\int_{(\varepsilon)} \widetilde{\Phi}(1-\alpha - \tfrac{s}{2}) X^{ - \frac{s}{2}} (la^2)^{\alpha +s}  \frac{G(s)}{s} \zeta_2(1 -2\alpha -2s) \frac{\Gamma \leg{\half - \alpha  - s}{2}}{\Gamma\leg{\thalf + \alpha}{2}}
 \leg{8}{\pi}^{-\alpha-\frac{s}{2}}  ds.
\end{equation}

Next note
\begin{equation} 
\leg{8}{\pi}^{-\alpha - \tfrac{s}{2}} \frac{\Gamma \leg{\half - \alpha -s}{2}}{\Gamma\leg{\thalf + \alpha}{2}} = \gamma_{\alpha} g_{-\alpha}(-s),
\end{equation}
so we get
\begin{equation}
 I= \gamma_{\alpha}
\frac{1}{2\pi i}\int_{(\varepsilon)} \widetilde{\Phi}(1-\alpha - \tfrac{s}{2}) X^{ - \frac{s}{2}} (la^2)^{ \alpha +s}  \frac{G(s)}{s} g_{-\alpha}(-s)
 \zeta_2(1 -2\alpha -2s) ds.
\end{equation}
Doing the change of variables $s \rightarrow -s$ gives
\begin{multline}
 M_N(k_1 = 1) = -\frac{X^{1-\alpha} \phi(l)/l}{2 l^{1/2}}  \gamma_{\alpha} \sum_{\substack{(a,2l) = 1 \\ a \leq Y}} \frac{\mu(a) \phi(a)/a}{a^2 \zeta_{2al}(2)} 
\\
\frac{1}{2\pi i}\int_{(-\varepsilon)} \widetilde{\Phi}(1-\alpha + \tfrac{s}{2}) X^{  \frac{s}{2}} (la^2)^{ \alpha -s} 
\frac{G(s)}{s} g_{-\alpha}(s)
 \zeta_2(1-2\alpha +2s) 
 ds,
\end{multline}
as desired.

We get $M_{-N}(k_1=1)$ from the previous expression using Remark \ref{remark:M1toM2}.  Precisely,
\begin{multline}
 M_{-N}(k_1 = 1) = -\frac{X \phi(l)/l}{2 l^{1/2}}   \sum_{\substack{(a,2l) = 1 \\ a \leq Y}} \frac{\mu(a) \phi(a)/a}{a^2 \zeta_{2al}(2)} 
\frac{1}{2\pi i}\int_{(-\varepsilon)} \widetilde{\Phi}(1 + \tfrac{s}{2}) X^{  \frac{s}{2}} (la^2)^{ - \alpha -s} \\ 
\frac{G(s)}{s} g_{\alpha}(s)
 \zeta_2(1 + 2\alpha +2s) 
 ds.
\end{multline}
Next note
\begin{equation}
\frac{1}{\zeta_{2al}(2)} \frac{\phi(l)}{l} \frac{\phi(a)}{a} = \frac{1}{\zeta_2(2)} \prod_{p|a} (1+p^{-1})^{-1} \prod_{p|l}(1+p^{-1})^{-1},
\end{equation}
so we get
\begin{multline}
 M_{-N}(k_1 = 1) = -\frac{X }{2 \zeta_2(2) l^{1/2}} \prod_{p|l} (1+p^{-1})^{-1}   \sum_{\substack{(a,2l) = 1 \\ a \leq Y}} \frac{\mu(a) }{a^2} \prod_{p|a}(1+p^{-1})^{-1}
\\
\frac{1}{2\pi i}\int_{(-\varepsilon)} \widetilde{\Phi}(1 + \tfrac{s}{2}) X^{  \frac{s}{2}} (la^2)^{-s - \alpha}  
\frac{G(s)}{s} g_{\alpha}(s)
 \zeta_2(1+2s + 2\alpha) 
 ds,
\end{multline}
which completes the proof.
\end{proof}

\section{Combining the main terms}
\label{section:mainterms}
\subsection{Gathering the terms}
Recall we have
\begin{multline}
\label{thm:combine}
 M(\alpha,l) = M_N(k=0) + M_{-N}(k = 0) + M_N(k_1 = 1) + M_{-N}(k_1 = 1)
\\
+ M_{R1} + M_{R2} + M_{-R1} + M_{-R2} + O\lp \frac{X^{f + \varepsilon}}{Y^{2f-1}} l^{1/2 + \varepsilon} + X^{1/4 + \varepsilon} Y l^{1/2 + \varepsilon} \rp.
\end{multline}
The rest of the paper is devoted to proving
\begin{mytheo}
 We have
\begin{equation}
\label{eq:MT1}
 M_N(k=0) + M_{-N}(k_1=1) + M_{R1} + M_{-R2} = \frac{X \widetilde{\Phi}(1)}{2 \zeta_2(2)} l^{-1/2 - \alpha} \zeta_{2}(1 + 2\alpha) B_{\alpha}(l),
\end{equation}
and
\begin{equation}
\label{eq:MT2}
 M_{-N}(k=0) + M_{N}(k_1=1) + M_{-R1} + M_{R2} =  \frac{X^{1-\alpha} \widetilde{\Phi}(1-\alpha)}{2 \zeta_2(2)} l^{-1/2 + \alpha} \zeta_{2}(1 - 2\alpha) B_{-\alpha}(l) \gamma_{\alpha}.
\end{equation}
\end{mytheo}
It suffices to show \eqref{eq:MT1} holds since \eqref{eq:MT2} follows by Remark \ref{remark:M1toM2}.

Before proving Theorem \ref{thm:combine}, we show how it proves Theorem \ref{theo:recursive}, and hence Conjecture \ref{conj:twistedconj}.  By taking $Y = X^{\frac{1}{4}}$, we get that if $M(\alpha,l) = M.T. + O(X^{f + \varepsilon} l^{1/2 + \varepsilon})$ then
\begin{equation}
M(\alpha,l) = M.T. + O( X^{\frac{f + \half}{2} + \varepsilon} l^{1/2 + \varepsilon}),
\end{equation}
where $M.T.$ is the main term from Conjecture \ref{conj:twistedconj}.

\subsection{Proving Theorem \ref{thm:combine}}
We first note that $M_N(k=0)$ given by \eqref{eq:P1} and $M_{R1}$ given by \eqref{eq:MR1b} combine naturally to give
\begin{multline}
M_{N}(k=0) + M_{R1} =  \frac{X \phi(l)/l}{2 l^{1/2 + \alpha}} \sum_{\substack{(a,2l) = 1}} \frac{\mu(a)}{a^2}    
\\
\frac{1}{2\pi i} \int_{(\varepsilon)} \widetilde{\Phi}(1 + \tfrac{s}{2})\frac{G(s)}{s} g_{\alpha}(s) X^{\frac{s}{2}} l^{-s} \frac{\zeta_{2a}(1 + 2\alpha + 2s)}{\zeta_{2al}(2 + 2\alpha + 2s)} ds
\end{multline}
We now compute the required Dirichlet series, namely
\begin{align}
 \sum_{(a,2l)=1} \frac{\mu(a)}{a^2} \frac{\zeta_{2a}(1 + 2\alpha + 2s)}{\zeta_{2al}(2 + 2\alpha + 2s)} &= \frac{\zeta_2(1 + 2\alpha + 2s)}{\zeta_{2l}(2 + 2\alpha + 2s)} \sum_{(a,2l)=1} \frac{\mu(a)}{a^2} \prod_{p | a} \frac{(1-p^{-1-2\alpha-2s})}{(1-p^{-2-2\alpha-2s})}
\\
& = \zeta_2(1 + 2\alpha + 2s) \prod_{p \nmid 2l} (1 - p^{-2-2\alpha-2s} - p^{-2} + p^{-3-2\alpha-2s}).
\end{align}
Then we use \eqref{eq:Balpha} to get
\begin{equation}
 \frac{\phi(l)}{l} \prod_{p \nmid 2l} (1 - p^{-2-2\alpha-2s} - p^{-2} + p^{-3-2\alpha-2s}) = \frac{1}{\zeta_2(2)} B_{\alpha+s}(l).
\end{equation}
That is, we have
\begin{multline}
\label{eq:right}
M_{N}(k=0) + M_{R1} =  \frac{X}{2 \zeta_2(2) l^{1/2 + \alpha}}   
\\
\frac{1}{2\pi i} \int_{(\varepsilon)} \widetilde{\Phi}(1 + \tfrac{s}{2})\frac{G(s)}{s} g_{\alpha}(s) X^{\frac{s}{2}} l^{-s} \zeta_2(1 + 2\alpha + 2s) B_{\alpha+s}(l)  ds.
\end{multline}

Now we work with $M_{-N}(k_1 = 1)$ and $M_{-R2}$ given by \eqref{eq:M-Nk1=1} and \eqref{eq:M-R2b}, respectively.  Clearly
\begin{multline}
M_{-N}(k_1 = 1) +M_{-R2} = -\frac{X }{2 \zeta_2(2) l^{1/2}} \prod_{p|l} (1+p^{-1})^{-1}   \sum_{\substack{(a,2l) = 1}} \frac{\mu(a) }{a^2} \prod_{p|a}(1+p^{-1})^{-1}
\\
\frac{1}{2\pi i}\int_{(-\varepsilon)} \widetilde{\Phi}(1 + \tfrac{s}{2}) \frac{G(s)}{s} g_{\alpha}(s) X^{  \frac{s}{2}} (la^2)^{- \alpha -s}  
 \zeta_2(1+ 2\alpha +2s) 
 ds.
\end{multline}
Then we calculate the sum over $a$ as follows:
\begin{equation}
 \sum_{(a,2l)=1} \frac{\mu(a)}{a^{2+2s+2\alpha }} \prod_{p|a} (1+p^{-1})^{-1} = \prod_{p \nmid 2l} (1- p^{-2-2\alpha-2s} (1+p^{-1})^{-1}),
\end{equation}
which equals
\begin{equation}
 B_{\alpha+s}(l) \prod_{p|l} (1+p^{-1}),
\end{equation}
in view of \eqref{eq:Balpha2}.  That is
\begin{multline}
\label{eq:left}
M_{-N}(k_1 = 1) +M_{-R2} 
\\
= -\frac{X }{2 \zeta_2(2) l^{1/2+\alpha}} 
\frac{1}{2\pi i}\int_{(-\varepsilon)} \widetilde{\Phi}(1 + \tfrac{s}{2}) \frac{G(s)}{s} g_{\alpha}(s) X^{  \frac{s}{2}} l^{-s}  
 \zeta_2(1+ 2\alpha +2s ) B_{\alpha+s}(l)
 ds.
\end{multline}

Noting that the integrands in \eqref{eq:right} and \eqref{eq:left} are identical, we see that the the sum of these two expressions is the sum of residues inside the strip $-\varepsilon \leq \text{Re}(s) \leq \varepsilon$.  The only residue is at $s=0$, by Remark \ref{remark:zero}.  It is obvious that the residue at $s=0$ is the main term of \eqref{eq:MT1}, which completes the proof.

\end{document}